\documentclass[12pt,a4paper,reqno,twoside]{amsart}

\usepackage[english]{babel}
\usepackage{stmaryrd}
\usepackage{dsfont}
\usepackage[symbol*,ragged]{footmisc}
\usepackage[colorlinks,linkcolor=red,anchorcolor=blue,citecolor=blue,urlcolor=blue]{hyperref}
\usepackage{color,xcolor}

\usepackage{geometry}
\usepackage{amssymb}
\usepackage{amsmath}
\usepackage{mathrsfs}
\usepackage{amsfonts}
\usepackage{epsfig}

\usepackage{amsthm}
\usepackage{amsxtra}
\usepackage{bbding}
\usepackage{epsfig}
\usepackage{graphicx}
\usepackage{latexsym}
\usepackage{mathbbol}
\usepackage{bbold}

\usepackage{pifont}
\usepackage{wasysym}
\usepackage{skull}
\usepackage{float}

\DeclareSymbolFontAlphabet{\mathbb}{AMSb}
\DeclareSymbolFontAlphabet{\mathbbol}{bbold}

\usepackage{amscd}
\usepackage[all]{xy}
\allowdisplaybreaks[4]
\usepackage{setspace}

\theoremstyle{plain}
\newtheorem{theorem}{\normalfont\scshape Theorem}[section]
\newtheorem{proposition}{\normalfont\scshape Proposition}[section]
\newtheorem{lemma}[proposition]{\normalfont\scshape Lemma}

\newtheorem*{corollary*}{\normalfont\scshape Corollary}

\theoremstyle{remark}
\newtheorem*{remark*}{\normalfont\scshape Remark}
\newtheorem*{notation}{\normalfont\scshape Notation}

\numberwithin{equation}{section}
\addtocounter{footnote}{1}

\renewcommand{\footnoterule}{
  \kern -3pt
  \hrule width 2.5in height 0.4pt
  \kern 3pt
}

\makeatletter
\@ifundefined{MakeUppercase}{}{}
\makeatother

\begin{document}
	
\title[ On a variant of the Piatetski--Shapiro prime number theorem ]
	  { On a variant of the Piatetski--Shapiro prime number theorem }

\author[Linji Long, Jinjiang Li, Min Zhang]
       {Linji Long \quad \& \quad Jinjiang Li \quad \& \quad Min Zhang}

\address{[Linji Long] Department of Mathematics, China University of Mining and Technology,
         Beijing 100083, People's Republic of China}

\email{linji.long.math@gmail.com}

\address{[Jinjiang Li] (Corresponding author) Department of Mathematics, China University of Mining and Technology,
         Beijing 100083, People's Republic of China}

\email{jinjiang.li.math@gmail.com}

\address{[Min Zhang] School of Applied Science, Beijing Information Science and Technology University,
		 Beijing 100192, People's Republic of China  }

\email{min.zhang.math@gmail.com}

\date{}

\footnotetext[1]{Jinjiang Li is the corresponding author. 

\quad\,\,
{\textbf{MR(2020) Subject Classification}}: 11N05, 11N32, 11N37, 11N80

}

\begin{abstract}
Denote by $\Pi_\gamma(x)$ the number of primes $p \leqslant x$ such that $[p^{1/\gamma}]$ is also a prime. In this paper, it is proved that the expected asymptotic formula
\begin{equation*}
\Pi_\gamma(x)=\frac{\gamma x}{(\log x)^2} (1+o(1))
\end{equation*}
holds for $\gamma > 15/13$, which constitutes an improvement upon the previous result of Balog \cite{Balog-1987} who showed that the above asymptotic formula holds for $\gamma > 6/5$.
\end{abstract}

\maketitle

\section{Introduction and main result}

Let $\gamma\in(\frac{1}{2},1)$ be a fixed real number. The Piatetski--Shapiro sequences are sequences of the form
\begin{equation*}
 \mathscr{N}_{\gamma}:=\big\{[ n^{1/\gamma}]:\,n\in \mathbb{N}^+\big\}.
\end{equation*}
Such sequences have been named in honor of Piatetski--Shapiro, who \cite{Piatetski-Shapiro-1953}, in
1953, proved that $\mathscr{N}_{\gamma}$ contains infinitely many primes provided that $\gamma\in(\frac{11}{12},1)$. The prime numbers of the form $p=[ n^{1/\gamma}]$ are called \textit{Piatetski--Shapiro primes of type $\gamma$}. More precisely, for such $\gamma$ Piatetski--Shapiro \cite{Piatetski-Shapiro-1953} showed that the counting function
\begin{equation*}
 \pi_\gamma(x):=\#\big\{\textrm{prime}\,\, p\leqslant x:\,p=[ n^{1/\gamma}]\,\,\textrm{for some}\,\,
 n\in\mathbb{N}^+ \big\}
\end{equation*}
satisfies the asymptotic property
\begin{equation}\label{pi-gamma-PNT}
\pi_{\gamma}(x)=\frac{x^{\gamma}}{\log x}(1+o(1))
\end{equation}
as $x\to\infty$. Since then, the range for $\gamma$ of the above asymptotic formula in which it is known that $\mathscr{N}_{\gamma}$ contains infinitely many primes has been enlarged many times (see the literatures \cite{Baker-Harman-Rivat-1995,Kolesnik-1967,Leitmann-1975,Leitmann-1980,Heath-Brown-1983,Kolesnik-1985,Liu-Rivat-1992,
Rivat-1992,Rivat-Sargos-2001}) over the years and is currently known to hold for all $\gamma\in(\frac{2426}{2817},1)$ thanks to Rivat and Sargos \cite{Rivat-Sargos-2001}. Rivat and Wu \cite{Rivat-Wu-2001} also showed that there exist infinitely many Piatetski--Shapiro primes for $\gamma\in(\frac{205}{243},1)$ by showing a lower bound of $\pi_\gamma(x)$ with the expected order of magnitude. We remark that if $\gamma>1$ then $\mathscr{N}_\gamma$ contains all natural numbers, and hence all primes, particularly.

For an analogue of (\ref{pi-gamma-PNT}), one can consider the subset of $\mathscr{N}_\gamma$ and expect the counting function
\begin{equation*}
\Pi_\gamma(x):= \#\big\{\textrm{prime}\,\, p\leqslant x:\, [p^{1/\gamma}] = q, \, q \textrm{ is prime} \big\}
\end{equation*}
subject to
\begin{equation}\label{eq1_2}
\Pi_\gamma(x)=\frac{\gamma x}{(\log x)^2}(1+o(1)) .
\end{equation}
Here and below, the letters $p,q$ always denote primes. This is not trivial even for $\gamma>1$. Indeed the problem of representing primes by linear polynomials in prime variable is still not solved. One can therefore look on the investigation of (\ref{eq1_2}) for $\gamma>1$ as a progress towards the twin prime conjecture.
In 1992, Balog \cite{Balog-1987} showed that (\ref{eq1_2}) holds for $\gamma>6/5$.

In this paper, we shall continue to improve the result of Balog \cite{Balog-1987}, and establish the following theorem.
\begin{theorem}\label{Theorem-1}
(\ref{eq1_2}) holds for $\gamma>15/13$.
\end{theorem}

\begin{notation}
Throughout this paper, $x$ denotes a sufficiently large number; $\varepsilon$ is a sufficiently small positive number, which may be distinct in each occurrences. Let $p$, $q$, with or without subscripts, always denote a prime number. As usual, we use $\Lambda(n)$ and $\psi(n)$ to denote von Mangoldt's function and the Chebychev function, respectively. $f(x) \ll g(x)$ means that $f(x)=O(g(x))$; $f(x) \asymp g(x)$ means that $f(x) \ll g(x) \ll f(x)$.
\end{notation}

\section{Preliminary Lemmas}

\begin{lemma}\label{psi-expansion}
	For any $T, x\geqslant 3$, one has
	\begin{equation*}
		\psi(x)=x-\sum_{|\tau|<T} \frac{x^{\rho}}{\rho}
			+O\bigg(\frac{x (\log x)^2}{T}+\frac{x \log T}{T}+\log x\bigg),
	\end{equation*}
	where the summation runs over all non--trivial zeros $\rho=\beta +i \tau$ of Riemann zeta--function, counted according to multiplicity.
\end{lemma}
\begin{proof}
	See see the arguments on page 283 of Landau \cite{Landau-1912}.
\end{proof}

\begin{lemma}\label{Large-Seive}
	Suppose that $\chi$ is a Dirichlet character modulo $q$ and $\mathscr{A}_\chi$ is a finite set of complex numbers $s=\sigma+i t$. Let
	\begin{equation*}
		S(s, \chi)=\sum_{n=1}^N a_n \chi(n) n^{-s},
	\end{equation*}
	and let $T_0, T, \sigma_0, \delta$ be real numbers such that for each $\chi$
	\begin{align*}
			& T_0+\frac{\delta}{2} \leqslant t \leqslant T_0+T-\frac{\delta}{2} \hspace{2em} (s \in \mathscr{A}_\chi), \\
			& \sigma \geqslant \sigma_0 \hspace{10em} (s \in \mathscr{A}_\chi),
	\end{align*}
	and
\begin{equation*}
|t-t^{\prime}| \geqslant \delta \hspace{2em} (s \in \mathscr{A}_\chi,\,\,
s^{\prime} \in \mathscr{A}_\chi,\,\, s \neq s^{\prime}) .
\end{equation*}
Then one has
\begin{equation*}
\sum_{\chi\bmod q} \sum_{s \in \mathscr{A}_\chi} |S(s, \chi)|^2 \ll(q T+N)(\delta^{-1}+\log N)
\sum_{n=1}^N|a_n|^2 n^{-2 \sigma_0}\bigg(1+\log \frac{\log 2 N}{\log 2 n}\bigg) .
\end{equation*}
\end{lemma}
\begin{proof}
	See Theorem 7.6 of Montgomery \cite{Montgomery-1971}.
\end{proof}

\begin{lemma}\label{zero-free-region}
There exists an absolute constant $c>0$ such that $\zeta(s)\neq 0$ with $s=\sigma + i t$ subject to
\begin{equation*}
\sigma > 1-\frac{c}{(\log (|t|+2))^{2 / 3}(\log \log (|t|+2))^{1 / 3}} .
\end{equation*}
\end{lemma}
\begin{proof}
	See Corollary 11.4 of Montgomery \cite{Montgomery-1971}.
\end{proof}

\begin{lemma}\label{zero-density-estimate}
	Let $N(\sigma, T)$ denote the number of nontrivial zeros $\rho$ of $\zeta(s)$ with $\Re(\rho) \geqslant \sigma$ and $|\Im(\rho)| \leqslant T$. Then, one has
	\begin{equation*}
		N(\sigma, T) \ll T^{30(1-\sigma) / 13+o(1)}.
	\end{equation*}
\end{lemma}
\begin{proof}
	See Theorem 1.2 of Guth and Maynard \cite{Guth-Maynard-2026}.
\end{proof}

\begin{lemma}\label{Huxley-density-estimate}
	For $4/5\leqslant\sigma\leqslant 1$ and $T\geqslant 3$, one has
	\begin{equation*}
		N(\sigma,T)\ll T^{3(1-\sigma)/(3\sigma-1)}(\log T)^C,
	\end{equation*}
	where $C>0$ is an absolute constant.
\end{lemma}
\begin{proof}
	This is the classical Huxley zero--density estimate; e.g., Chapter 11 of Ivic \cite{Ivic-1985}.
\end{proof}

\section{Proof of Theorem \ref{Theorem-1}}

For $X\geqslant 3$, we put
\begin{equation*}
\mathcal{S}_{\gamma}(X):=
\sum_{p\leqslant X}\sum_{p^\gamma \leqslant q < (p+1)^\gamma}1.
\end{equation*}
Since $\gamma>1$, for primes $p$ and $q$ one has
\begin{equation*}
p=[q^{1/\gamma}]
\quad\iff\quad p\leqslant q^{1/\gamma}<p+1
\quad\iff\quad p^\gamma \leqslant q < (p+1)^\gamma.
\end{equation*}
Thus, for a prime $q\leqslant X^\gamma$, the condition that $[q^{1/\gamma}]$ is a prime is equivalent to the existence of a prime $p\leqslant X$ such that $p^\gamma\leqslant q<(p+1)^\gamma$. Therefore
\begin{equation*}
\Pi_\gamma(X^\gamma)=\mathcal{S}_{\gamma}(X)+O(X^{\gamma-1}).
\end{equation*}
Then, it suffices to prove that
\begin{equation}\label{eq3_0}
\mathcal{S}_{\gamma}(X)=\frac{X^\gamma}{\gamma(\log X)^2}(1+o(1)).
\end{equation}
Hence, by Taylor's formula, we obtain
\begin{align*}
	\mathcal{S}_{\gamma}(X)
	= & \,\,  \sum_{p\leqslant X} \sum_{p^\gamma \leqslant q < (p+1)^\gamma} 1
	=  \,\, \sum_{p\leqslant X}
			\Big(\pi((p+1)^\gamma) - \pi(p^{\gamma}) +O(1)\Big)
					\\
	= & \,\, \sum_{p\leqslant X}
			\Big(\pi(p^\gamma + \gamma p^{\gamma-1}) - \pi(p^{\gamma})\Big)
			+ O\Bigg(\sum_{p\leqslant X} (p^{\gamma-2}+1)\Bigg)
					\\
	= & \,\,  \sum_{p\leqslant X}
		\Big(\pi(p^\gamma + \gamma p^{\gamma-1}) - \pi(p^{\gamma})\Big)
			+ O\bigg(\frac{X}{\log X} + \frac{X^{\gamma-1}}{\log X}\bigg).
\end{align*}
By a splitting argument, we deduce that
\begin{align}\label{eq3_1}
	\mathcal{S}_{\gamma}(X)
	= & \,\, \sum_{\mathcal{Q}} \sum_{Q<p \leqslant Q+Q/\log Q}
	\Big(\pi(p^\gamma + \gamma p^{\gamma-1} ) - \pi(p^{\gamma})\Big)
	+O\bigg(\frac{X}{\log X}
	+\frac{X^{\gamma-1}}{\log X}\bigg)
			\nonumber\\
	= & \,\, \sum_{\mathcal{Q}} \sum_{Q<p \leqslant Q+Q/\log Q}
			\Big(\pi(p^\gamma + \gamma p^{\gamma} Q^{-1}) - \pi(p^{\gamma})\Big)
			\nonumber\\
	  & \,\, +O\bigg(\mathcal{E}(X)+\frac{X}{\log X}
			+\frac{X^{\gamma-1}}{\log X}\bigg),
\end{align}
where
\begin{equation*}
\mathcal{E}(X):=\sum_{\mathcal{Q}}\sum_{Q<p\leqslant Q+Q/\log Q}
\Big(\pi(p^\gamma+\gamma p^\gamma Q^{-1})
-\pi(p^\gamma+\gamma p^{\gamma-1})\Big),
\end{equation*}
and the first summation above in (\ref{eq3_1}) runs over a sequence of a certain set of integers $Q$ satisfying
$Q\leqslant X$ and
\begin{equation*}
	\sum_{\mathcal{Q}} 1 \ll \int_{2}^{X} \frac{1}{u/\log u} \mathrm{d}u \ll (\log X)^2.
\end{equation*}
Now, we shall establish explicit the upper bound estimate for primes, which is used to deal with
$\mathcal{E}(X)$, in short intervals. Set
\begin{equation*}
	H:=\gamma p^{\gamma-1}\frac{p-Q}{Q}
	\asymp Q^{\gamma-2}(p-Q),
	\qquad
	\vartheta:=\min\bigg\{\frac{\gamma-1}{2},\frac{1}{2}\bigg\}>0.
\end{equation*}
First, we consider the terms subject to $H\leqslant Q^\vartheta$, which implies
$p-Q\ll Q^{2-\gamma+\vartheta}$. If $Q^{2-\gamma+\vartheta}<1$, then the range is empty for all sufficiently large $Q$.
Otherwise, the trivial estimate $\pi(u+h)-\pi(u)\leqslant h+1$ gives that 
\begin{align*}
	&  \sum_{\substack{Q<p\leqslant Q+Q/\log Q\\ H\leqslant Q^\vartheta}}
	\Big(\pi(p^\gamma+\gamma p^\gamma Q^{-1})
	-\pi(p^\gamma+\gamma p^{\gamma-1})\Big)
				\\
	&\qquad\ll Q^{\gamma-2}
				\sum_{\substack{Q<p\leqslant Q+Q/\log Q\\ H\leqslant Q^\vartheta}}
				(p-Q) + \sum_{\substack{Q<p\leqslant Q+Q/\log Q\\ H\leqslant Q^\vartheta}} 1
				\\
	&\qquad\ll Q^{\gamma-2}(Q^{2-\gamma+\vartheta})^2+Q^{2-\gamma+\vartheta}
			\ll Q.
\end{align*}
Since $\sum_{\mathcal Q}1\ll(\log X)^2$, the total contribution of these terms is
\begin{equation*}
	\ll X(\log X)^2=o\bigg(\frac{X^\gamma}{(\log X)^3}\bigg).
\end{equation*}
For the remaining terms, we have $H>Q^\vartheta$. Since
$p\leqslant Q+Q/\log Q$ and $p\asymp Q$, one has
\begin{equation*}
	H=\gamma p^{\gamma-1}\frac{p-Q}{Q}
	\ll \frac{Q^{\gamma-1}}{\log Q}
	=o(Q^\gamma).
\end{equation*}
By setting $u=p^\gamma+\gamma p^{\gamma-1}$, we have
$u\asymp Q^\gamma$. Hence, for all sufficiently large $Q$, we derive that 
\begin{equation*}
	2\leqslant H\leqslant u.
\end{equation*}
In this case, there holds the Brun--Titchmarsh inequality
\begin{equation*}
	\pi(u+h)-\pi(u)\leqslant\frac{2h}{\log h}
	\qquad (2\leqslant h\leqslant u)
\end{equation*}
 with $h=H$. Moreover, since
$H>Q^\vartheta$, we have $\log H\geqslant\vartheta\log Q$.
By noting the estimate  
\begin{align*}
	\sum_{Q<p \leqslant Q+Q/\log Q} (p-Q)
		= & \,\, \int_{Q}^{Q+Q/\log Q} (u-Q) \mathrm{d} \pi(u)
			= (1+o(1)) \int_{Q}^{Q+Q/\log Q} \frac{u-Q}{\log u} \mathrm{d} u
				\\
		\ll & \,\, \frac{1}{\log Q} \int_{Q}^{Q+Q/\log Q} (u-Q) \mathrm{d} u
			\ll \frac{Q^2}{(\log Q)^3},
\end{align*}
we deduce that
\begin{align*}
	&\sum_{\mathcal Q}\sum_{\substack{Q<p\leqslant Q+Q/\log Q\\H>Q^\vartheta}}
	\Big(\pi(p^\gamma+\gamma p^\gamma Q^{-1})
	-\pi(p^\gamma+\gamma p^{\gamma-1})\Big)\\
	&\qquad\ll
	\sum_{\mathcal Q}\frac{Q^{\gamma-2}}{\log Q}
	\sum_{Q<p\leqslant Q+Q/\log Q}(p-Q)\\
	&\qquad\ll \sum_{\mathcal Q}\frac{Q^\gamma}{(\log Q)^4}
	\ll \int_2^X\frac{u^\gamma}{(\log u)^4}
	\cdot\frac{1}{u/\log u}\,\mathrm{d}u
	\ll\frac{X^\gamma}{(\log X)^3}.
\end{align*}
Consequently, we obtain
\begin{equation*}
	\mathcal{E}(X)\ll\frac{X^\gamma}{(\log X)^3},
\end{equation*}
which combined with (\ref{eq3_1}) yields that 
\begin{equation}\label{eq3_2}
	\mathcal{S}_{\gamma}(X)
	= \sum_{\mathcal{Q}} \sum_{Q<p \leqslant Q+Q/\log Q}
		\Big(\pi(p^\gamma + \gamma p^{\gamma} Q^{-1}) - \pi(p^{\gamma})\Big)
	+ O\bigg(\frac{X^\gamma}{(\log X)^3}\bigg).
\end{equation}
Now, we focus on the main term of (\ref{eq3_2}). For $y=o(z)$, it is easy to see that
\begin{align*}
	& \,\,\pi(z+y) - \pi(z) = \sum_{z<p\leqslant z+y} 1
				\\
	= & \,\, \sum_{z<n\leqslant z+y}\frac{\Lambda(n)}{\log n}
		+ O\bigg(\sum_{2\leqslant k\ll \log z}\sum_{z<p^k\leqslant z+y}
			\frac{1}{k}\bigg)
				\\
	= & \,\, \sum_{z<n\leqslant z+y}\frac{\Lambda(n)}{\log z}
		+ O\bigg(\frac{y}{(\log z)^2} +
				\sum_{2\leqslant k\ll \log z}\sum_{z<p^k\leqslant z+y}
				\frac{1}{k}\bigg)
				 \\
	= & \,\, \frac{\psi(z+y) - \psi(z)}{\log z}
			+ O\bigg(\frac{y}{(\log z)^2} +
				\sum_{2\leqslant k\ll \log z}\sum_{z<p^k\leqslant z+y}
				\frac{1}{k}\bigg).
\end{align*}
Thus, we obtain that
\begin{align}\label{eq3_3}
	  & \,\, \sum_{\mathcal{Q}} \sum_{Q<p \leqslant Q+Q/\log Q}
			\Big(\pi(p^\gamma + \gamma p^{\gamma} Q^{-1}) - \pi(p^{\gamma})\Big)
					\nonumber\\
	= & \,\, \sum_{\mathcal{Q}} \sum_{Q<p \leqslant Q+Q/\log Q}
			\frac{1}{\gamma \log p}
			\Big(\psi(p^\gamma + \gamma p^{\gamma} Q^{-1}) - \psi(p^{\gamma})\Big)
					\nonumber\\
	  & \,\, + O\Bigg(\sum_{\mathcal{Q}} \sum_{Q<p \leqslant Q+Q/\log Q}
	  		\frac{p^{\gamma} Q^{-1}}{(\log Q)^2}
	  		+ \sum_{\mathcal{Q}} \sum_{Q<p \leqslant Q+Q/\log Q}
	  		\,\sum_{2\leqslant k\ll \log p}
	  		\,\sum_{p^\gamma <q^k\leqslant p^\gamma+\gamma p^\gamma Q^{-1}}
			\frac{1}{k}\Bigg).
\end{align}
The first sum in the error term of (\ref{eq3_3}) is bounded by
\begin{align*}
	\sum_{\mathcal{Q}} \sum_{Q<p \leqslant Q+Q/\log Q}
		\frac{p^{\gamma} Q^{-1}}{(\log Q)^2}
	= & \,\, \sum_{\mathcal{Q}} \frac{1}{Q(\log Q)^2}
		\int_Q^{Q+Q/\log Q} u^{\gamma} \mathrm{d} \pi(u)
			\\
	= & \,\, \sum_{\mathcal{Q}}	\frac{1+o(1)}{Q(\log Q)^2}
		\int_Q^{Q+Q/\log Q} \frac{u^{\gamma}}{\log u} \mathrm{d} u
			\\
	\ll & \,\, \sum_{\mathcal{Q}}	\frac{1+o(1)}{Q(\log Q)^3}
		\cdot \bigg(\bigg(Q+\frac{Q}{\log Q}\bigg)^{\gamma+1}- Q^{\gamma+1}\bigg)
			\\
	\ll & \,\, \sum_{\mathcal{Q}}\frac{Q^\gamma}{(\log Q)^4}
	\ll  \frac{X^\gamma}{(\log X)^3}.
\end{align*}
For the second sum in the error term of (\ref{eq3_3}), by noting that for distinct primes $p$ the intervals $(p^{\gamma}, p^{\gamma} + \gamma p^{\gamma}Q^{-1}]$ are pairwise disjoint, and thus
\begin{align*}
	\sum_{\mathcal{Q}} \sum_{Q<p \leqslant Q+Q/\log Q}
		\,\sum_{2\leqslant k\ll \log p}
		\,\sum_{p^\gamma <q^k\leqslant p^\gamma+\gamma p^\gamma Q^{-1}}
		\frac{1}{k}
	\ll \sum_{2\leqslant k\ll \log X} \, \sum_{q^k \leqslant X^\gamma} \frac{1}{k}
	\ll X^{\gamma/2+\varepsilon} \ll \frac{X^\gamma}{(\log X)^3}.
\end{align*}
Hence, we derive that
\begin{align}\label{eq3_4}		
	& \,\, \sum_{\mathcal{Q}} \sum_{Q<p \leqslant Q+Q/\log Q}
		\Big(\pi(p^\gamma + \gamma p^{\gamma} Q^{-1}) - \pi(p^{\gamma})\Big)
			\nonumber\\
	= & \,\, \sum_{\mathcal{Q}} \sum_{Q<p \leqslant Q+Q/\log Q}
		\frac{1}{\gamma \log p}
		\Big(\psi(p^\gamma + \gamma p^{\gamma} Q^{-1}) - \psi(p^{\gamma})\Big)
		+ O\bigg(\frac{X^\gamma}{(\log X)^3}\bigg).
\end{align}
Each term in the inner sum of (\ref{eq3_4}) can be expressed by the zeros of Riemann $\zeta$--function via the well--known explicit formula for $\psi(y)$, e.g., Lemma \ref{psi-expansion}. For $Q^\gamma \leqslant y \leqslant 2Q^{\gamma}$, there holds
\begin{equation}\label{psi-in}
	\psi(y)
	= y -\sum_{|\tau| \leqslant T} \frac{y^\rho}{\rho}
		+ O\bigg(\frac{Q^{\gamma} (\log Q)^2}{T}+\frac{Q^{\gamma} \log T}{T}+\log Q\bigg)
\end{equation}
uniformly for $Q,T\geqslant 3$. Here the summation runs over all non--trivial zeros $\rho=\beta +i \tau$ of $\zeta(s)$ counted according to multiplicity. The almost optimal choice of $T$ is given by $T=Q (\log Q)^5$, which takes the error term being small enough. It follows from (\ref{eq3_2}), (\ref{eq3_4}), (\ref{psi-in}) and partial summation that
\begin{align*}
	\mathcal{S}_{\gamma}(X)
	= & \,\, \sum_{\mathcal{Q}} \sum_{Q<p \leqslant Q+Q/\log Q}
			\frac{p^{\gamma}}{Q \log p}
		+ O\bigg(\frac{X^\gamma}{(\log X)^3}\bigg)
				\\
	  & \,\, - \sum_{\mathcal{Q}} \sum_{Q<p \leqslant Q+Q/\log Q}
			\frac{1}{\gamma \log p} \sum_{|\tau| \leqslant T}
			\frac{(p^\gamma + \gamma p^{\gamma} Q^{-1})^\rho-(p^\gamma)^\rho}{\rho}
				\\
	= & \,\, \sum_{p\leqslant X}\frac{p^{\gamma-1}}{\log p}
			+ O\bigg(\sum_{\mathcal{Q}} \sum_{Q<p \leqslant Q+Q/\log Q}
				\frac{p^{\gamma-1}}{\log p}\cdot\frac{(p-Q)}{Q}\bigg)
			+ O \bigg(\frac{X^\gamma}{(\log X)^3}\bigg)
				\\
	  & \,\, - \sum_{\mathcal{Q}} \sum_{Q<p \leqslant Q+Q/\log Q}
			\frac{1}{\gamma \log p} \sum_{|\tau| \leqslant T}
			\frac{(p^\gamma + \gamma p^{\gamma} Q^{-1})^\rho-(p^\gamma)^\rho}{\rho}
				\\
	= & \,\, \frac{X^\gamma}{\gamma(\log X)^2} + 	
			O\bigg(\frac{X^\gamma}{(\log X)^3}\bigg)
			+ O\Bigg(\sum_{\mathcal{Q}} \frac{|R(Q)|}{\log Q}\Bigg),
\end{align*}
where
\begin{align*}
	R(Q) :=
	& \,\,	\sum_{Q<p \leqslant Q+Q/\log Q}
			\sum_{|\tau| \leqslant T}
			\frac{(p^\gamma + \gamma p^{\gamma} Q^{-1})^\rho-(p^\gamma)^\rho}{\rho}
					\\
	= & \,\, \sum_{|\tau| \leqslant T} \frac{(1+\gamma Q^{-1})^{\rho}-1}{\rho}
			\sum_{Q<p \leqslant Q+Q/\log Q} p^{\gamma\rho},
\end{align*}
and the main term follows from partial summation and prime number theorem demonstrated as follows
	\begin{align*}
		\sum_{p\leqslant X}\frac{p^{\gamma-1}}{\log p}
		& =\int_2^X \frac{u^{\gamma-1}}{\log u}\,\mathrm{d}\pi(u) =\int_2^X \frac{u^{\gamma-1}}{(\log u)^2}\,\mathrm{d}u
		+O\bigg(\frac{X^\gamma}{(\log X)^3}\bigg)
		\\
		& =\frac{X^\gamma}{\gamma(\log X)^2}+O\bigg(\frac{X^\gamma}{(\log X)^3}\bigg).
	\end{align*}

Define
\begin{equation*}
	F(s)=\sum_{Q<p \leqslant Q+Q/\log Q} p^{-s},\qquad  \quad C(s)=\frac{(1+\gamma Q^{-1})^s-1}{s}.
\end{equation*}
Then, we have
\begin{equation*}
	R(Q) = \sum_{|\tau| \leqslant T} C(\rho)F(-\gamma\rho).
\end{equation*}
It suffices to show that
\begin{equation}\label{eq3_5}
	R(Q) \ll \frac{X^\gamma}{(\log X)^5}
\end{equation}
holds uniformly for $Q\leqslant X$.
According to the real part of the zeros of $\zeta(s)$, we decompose $R(Q)$  as follows
\begin{equation*}
	R(Q)=R_1(Q)+R_2(Q),
\end{equation*}
where
\begin{equation*}
	R_1(Q):=\sum_{\substack{\beta\geqslant1/2\\|\tau|\leqslant T}}
	C(\rho)F(-\gamma\rho),
	\qquad
	R_2(Q):=\sum_{\substack{\beta<1/2\\|\tau|\leqslant T}}
	C(\rho)F(-\gamma\rho).
\end{equation*}
We first deal with $R_1(Q)$. Divide the interval $[1/2,1]$ into $J$ disjoint subintervals
$I_j=[\sigma_j,\sigma_{j+1})$, $(j=1,2,\dots,J)$, with length 
$\Delta\sigma_j:=\sigma_{j+1}-\sigma_j\asymp1/\log Q$, where the number of such intervals satisfies $J\ll\log Q$.
By the trivial upper bound $|C(\rho)|\ll1/Q$ and Cauchy's inequality,
the contribution of $R_1(Q)$ satisfies
\begin{align}\label{eq3_6}
	R_1(Q)
	\ll & \,\, \frac{1}{Q} \sum_{j=1}^{J}
		\sum_{\substack{\beta \in I_j \\ |\tau| \leqslant T}} |F(-\gamma\rho)|
				\nonumber\\
	\ll & \,\, \frac{1}{Q} \sum_{j=1}^{J}
			\Bigg(\sum_{\substack{\beta \geqslant \sigma_j \\
						|\tau| \leqslant T}} 1 \Bigg)^{1/2}
			\Bigg(\sum_{\substack{\beta \in I_j \\ |\tau| \leqslant T}}
			|F(-\gamma\rho)|^2 \Bigg)^{1/2}
				\nonumber\\
	\ll & \,\, \frac{1}{Q} \sum_{j=1}^{J}
			\big(N(\sigma_j, T)\big)^{1/2}
			\Bigg(\sum_{\substack{\beta \in I_j \\ |\tau| \leqslant T}}
			|F(-\gamma\rho)|^2 \Bigg)^{1/2}.
\end{align}
where $N(\sigma_j, T)$ denotes the number of nontrivial zeros $\rho=\beta+i \tau$ of $\zeta(s)$ satisfying $\beta \geqslant \sigma_j$ and $|\tau| \leqslant T$.
For the second factor, we divide the zeros with $\beta\in I_j$ and $|\tau|\leqslant T$ into several subsets. By the Riemann--von Mangoldt formula, the number of zeros with $u<|\tau|\leqslant u+1$ is $O(\log (u+2))$ (e.g., see Chapter 1 of Ivi\'c \cite{Ivic-1985}). For each integer $m$, we arrange the zeros with $m<\tau\leqslant m+1$ in increasing order based on the magnitude of their imaginary part, and assign them a series of subscript. Then for each fixed subscript $r$, we gather the $r$--th zeros from the intervals $(m,m+1]$ with $m$ being even numbers, and similarly those with $m$ being odd numbers. Since each unit interval contains $O(\log T)$ zeros, this yields $O(\log T)$ subsets. Hence, these zeros, which satisfy $\beta\in I_j$ and $|\tau|\leqslant T$, can be classified as the union of $O(\log T)$ subsets such that any two distinct zeros in the same subset ssatisfy $|\tau-\tau'|\geqslant 1$. For each subset, we apply Lemma \ref{Large-Seive} to $F(s)$ at the points $s=-\gamma\rho$ with $q=1$ and $\delta=1$, where the range of the summation is $\ll Q$, and $\Re(-\gamma\rho)=-\gamma\beta\geqslant-\gamma\sigma_{j+1}$. Since $p\asymp Q$ and $\sigma_{j+1}-\sigma_j\asymp1/\log Q$, the factor $p^{2\gamma\sigma_{j+1}}$ arising from Lemma \ref{Large-Seive} is $\asymp p^{2\gamma\sigma_j}$. Accordingly, after summing over the above $O(\log T)$ subsets, we derive that
\begin{equation*}
	\sum_{\substack{\beta \in I_j \\ |\tau| \leqslant T}}
	|F(-\gamma\rho)|^2
	\ll (\log Q)^2 (T+Q) \sum_{Q<p \leqslant Q+Q/\log Q} p^{2\gamma\sigma_{j}}.
\end{equation*}
Summing over all $j=1, \dots, J$, we obtain that
\begin{align}\label{eq3_7}
	R_1(Q)
	\ll & \,\, \frac{\log Q}{Q} \sum_{j=1}^{J} \big(N(\sigma_j, T)\big)^{1/2}
			\Bigg((T+Q) \sum_{Q<p \leqslant Q+Q/\log Q} p^{2\gamma\sigma_j}\Bigg)^{1/2}
					\nonumber \\
	\ll & \,\, \frac{(\log Q)^2}{Q}
			\sup_{1/2\leqslant\sigma\leqslant\sigma_0}
			\Bigg(\big(N(\sigma, T)\big)^{1/2} (T+Q)^{1/2}
			\Bigg(\sum_{Q<p \leqslant Q+Q/\log Q} p^{2\gamma\sigma} \Bigg)^{1/2} \Bigg)
					\nonumber \\
	\ll & \,\, (\log Q)^2 \sup_{1/2\leqslant\sigma\leqslant\sigma_0}
			\big(N(\sigma, T)\big)^{1/2} T^{1/2} Q^{-1/2+\gamma\sigma},
\end{align}
where
\begin{equation}\label{eq3_8}
	\sigma_0 = 1 - \frac{c}{(\log T)^{2 / 3} (\log \log T)^{1 / 3}}
\end{equation}
follows from Vinogradov's zero--free region, i.e., Lemma \ref{zero-free-region}.
Recalling that $T = Q(\log Q)^5$ and $Q \leqslant X$,
we consider the supremum in (\ref{eq3_7}) in two cases. If $1/2\leqslant\sigma\leqslant 4/5$, then we use Lemma \ref{zero-density-estimate} to get
\begin{equation}\label{eq3_9}
(\log Q)^2\big(N(\sigma,T)\big)^{1/2}T^{1/2}Q^{-1/2+\gamma\sigma}
\ll Q^{\gamma+(15/13-\gamma)(1-\sigma)+o(1)}
\ll Q^{\gamma-\eta_{\gamma}}
\ll X^{\gamma-\eta_{\gamma}}
\end{equation}
for some $\eta_\gamma>0$, provided that $\gamma>15/13$. If $4/5\leqslant\sigma\leqslant\sigma_0$, then Lemma \ref{Huxley-density-estimate} yields
\begin{align}\label{eq3_10}
 (\log Q)^2\big(N(\sigma,T)\big)^{1/2}T^{1/2}Q^{-1/2+\gamma\sigma}
\ll (\log Q)^{C_1} Q^{\gamma\sigma+\frac{3(1-\sigma)}{2(3\sigma-1)}}
\ll (\log Q)^{C_1} Q^{\gamma+(15/14-\gamma)(1-\sigma)},
\end{align}
where $C_1>0$ is absolute.
For $Q\leqslant X^{1/2}$, the last expression in (\ref{eq3_10}) is trivially
$\ll X^\gamma(\log X)^{-5}$. For $X^{1/2}<Q\leqslant X$, then $\log Q\asymp\log X$, and thus
\begin{equation*}
	Q^{(15/14-\gamma)(1-\sigma)}
	\leqslant \exp\big(-\mathfrak{c}_\gamma(\log Q)^{1/3}(\log\log Q)^{-1/3}\big)\ll (\log X)^{-A}
\end{equation*}
for any fixed $A>0$, provided that $\gamma>15/14$. Hence (\ref{eq3_10}) gives
$\ll X^\gamma(\log X)^{-5}$ in this case as well.
Combining the above two cases, we obtain
$|R_1(Q)|\ll X^\gamma(\log X)^{-5}$.

Next, we shall estimate $R_2(Q)$. By the same separation argument used above, the zeros with
$\beta<1/2$ and $|\tau|\leqslant T$ can be divided into $O(\log T)$ subsets such that
any two distinct zeros in the same subset satisfy $|\tau-\tau'|\geqslant1$. For each such subset, we apply Lemma \ref{Large-Seive} to $F(s)$ at the points $s=-\gamma\rho$, with $q=1$, $\delta=1$ and $\sigma_0=-\gamma/2$. Since
$\Re(-\gamma\rho)=-\gamma\beta>-\gamma/2$, we obtain
\begin{align*}
	\sum_{\substack{\beta<1/2\\|\tau|\leqslant T}}
	|F(-\gamma\rho)|^2
	&\ll(\log Q)^2(T+Q)
	\sum_{Q<p\leqslant Q+Q/\log Q}p^\gamma\\
	&\ll TQ^{\gamma+1}\log Q.
\end{align*}
On the other hand, it follows from the Riemann--von Mangoldt formula that 
\begin{equation*}
	\sum_{\substack{\beta<1/2 \\ |\tau|\leqslant T}} 1 \ll T\log T.
\end{equation*}
Since $\gamma>1$, therefore, by $|C(\rho)|\ll Q^{-1}$ and Cauchy's inequality, recalling that
$T=Q(\log Q)^5$, we deduce that 
\begin{align*}
	|R_2(Q)|
	\ll&\,\, \frac{1}{Q}\Bigg(\sum_{\substack{\beta<1/2 \\ |\tau|\leqslant T}} 1\Bigg)^{1/2}
	\Bigg(\sum_{\substack{\beta<1/2\\|\tau|\leqslant T}}
	|F(-\gamma\rho)|^2\Bigg)^{1/2}
			\\
	\ll&\,\, \frac{1}{Q}(T\log T)^{1/2}
	\Bigg(\sum_{\substack{\beta<1/2\\|\tau|\leqslant T}}
	|F(-\gamma\rho)|^2\Bigg)^{1/2}
			\\
	\ll&\,\, Q^{(\gamma+1)/2}(\log Q)^6
	\ll\frac{X^\gamma}{(\log X)^5},
\end{align*}
uniformly for $Q\leqslant X$.
Combining the estimates of $R_1(Q)$ and $R_2(Q)$, we prove (\ref{eq3_5}).
This completes the proof of (\ref{eq3_0}). Finally, by taking $x=X^\gamma$, we derive that
\begin{equation*}
	\Pi_\gamma(x)=\frac{X^\gamma}{\gamma(\log X)^2}(1+o(1))
		=\frac{\gamma x}{(\log x)^2}(1+o(1)),
\end{equation*}
which completes the proof of Theorem \ref{Theorem-1}.

\section*{Acknowledgement}

%\noindent
The authors would like to appreciate the referee for his/her patience in refereeing this paper.
This work is supported by Beijing Natural Science Foundation (Grant No. 1242003), and
the National Natural Science Foundation of China (Grant Nos. 12471009, 12301006, 11901566, 12001047).

\end{document}